\documentclass[12pt,leqno]{amsart}
\usepackage{amssymb}
\usepackage{esint}
\usepackage{mathtools}
\usepackage{hyperref}
\usepackage{cleveref}
\makeatletter
\newcommand*{\rom}[1]{\expandafter\@slowromancap\romannumeral #1@}
\makeatother
\title{T\MakeLowercase{he} \MakeLowercase{volume of the boundary of a} S\MakeLowercase{obolev} $(p,q)$-\MakeLowercase{extension domain} \rom{2}}
\author{Pekka Koskela, Riddhi Mishra}

\address{P.\ Koskela: Department of Mathematics and Statistics, University of Jyv\"askyl\"a, P.O. Box 35 (MaD), FI-40014, University of Jyv\"askyl\"a, Finland. {\tt pekka.j.koskela@jyu.fi }}

\address{R.\ Mishra: Department of Mathematics and Statistics, University of Jyv\"askyl\"a, P.O. Box 35 (MaD), FI-40014, University of Jyv\"askyl\"a, Finland. {\tt riddhi.r.mishra@jyu.fi }}

plus 6pt minus 12pt
\newtheorem{theorem}{Theorem}
\newtheorem{lemma}[theorem]{Lemma}
\newtheorem{corollary}[theorem]{Corollary}

\theoremstyle{definition}

\newtheorem{definition}[theorem]{Definition}

\newcommand{\barint}{
\rule[.036in]{.12in}{.009in}\kern-.16in \displaystyle\int }

\newcommand{\barcal}{\mbox{$ \rule[.036in]{.11in}{.007in}\kern-.128in\int $}}
\newcommand{\diam}{\mathop{\mathrm {diam}}}

\newcommand{\mbbr}{\mathbb R}

\def\diam{\operatorname{diam}}

\def\mvint_#1{\mathchoice
          {\mathop{\vrule width 6pt height 3 pt depth -2.5pt
                  \kern -8pt \intop}\nolimits_{\kern -3pt #1}}%
          {\mathop{\vrule width 5pt height 3 pt depth -2.6pt
                  \kern -6pt \intop}\nolimits_{#1}}%
          {\mathop{\vrule width 5pt height 3 pt depth -2.6pt
                  \kern -6pt \intop}\nolimits_{#1}}%
          {\mathop{\vrule width 5pt height 3 pt depth -2.6pt
                  \kern -6pt \intop}\nolimits_{#1}}}

\numberwithin{theorem}{section} \numberwithin{equation}{section}

\thanks{Both authors have been supported by the Academy of Finland via Centre of Excellence in Analysis and Dynamics
Research (Project number 323960)}
\begin{document}
\maketitle
\begin{abstract}
    We show that the volume of the boundary of a bounded Sobolev $(p,q)$-extension domain is zero when $1\leq q <p< \frac{qn}{(n-q)}.$
\end{abstract}
\section{Introduction}
Let $\Omega \subset \mbbr^n$ be a bounded $(W^{1,p},W^{1,q})$ -Sobolev extension domain with $1\leq q <p.$ This simply means that every function
in $W^{1,p}(\Omega)$ has an extension that belongs to $W^{1,q}(\mbbr^n),$
with norm controlled by the original norm.
In the case $q=p,$ the domain $\Omega$ is then necessarily locally Ahlfors
regular in the sense that
\begin{equation}\label{ahlfors}
 | B(x,r)\cap \Omega| \ge \delta{r}^n
\end{equation} holds for every $x \in \partial \Omega$ and each $0<r<1.$
Especially, the volume of $\partial\Omega$ must be zero. For this see
\cite{HKT} and also \cite{GOM},\cite{MR1039115}, \cite{SKV} for earlier
results. However, \eqref{ahlfors} does not need to hold in the case $1\leq
q<p$ as shown in the work of Maz'ya and Poborchi \cite{MVP}. Indeed, this
condition fails at tips of monomial exterior cusps, which are extension
domains for suitable values of $q$ and $p$. The best version of \eqref
{ahlfors} known up to now is  \begin{equation} \label{sahlfors}
 |B(x,r)\cap\Omega|\geq \delta r^{s}
\end{equation}
under the assumption that $q>n-1$ ($q \ge 1$ when $n=2$),
where $s>n$ depends on $q,p$ and $n,$ and tends to infinity for a fixed $q
$ when $p$ tends to infinity \cite{PUZ}. When $1\leq q\leq n-1$ with $n
\ge 3,$ one can not bound the decay order in general. On the other hand,
the validity of \eqref{sahlfors} for all $x \in \partial\Omega$ of a given
domain $\Omega$ does not rule out the possibility of $\partial \Omega$
having positive volume. Nevertheless, it was shown in \cite{PUZ} that $
\partial \Omega$ is necessarily of volume zero when $q>n-1.$ The bound $n-
1$ is sharp in the sense that examples of extension domains with positive
boundary volume for $n\geq 3$ were constructed in \cite{PUZ} for any value
of $1\leq q<n-1,$ for a suitable $p.$ The natural question of the
dependence of $p$ in terms of $q$ in such examples was also raised in
\cite[Question 8.3]{PUZ}.
Our main result shows that the volume of $\partial\Omega$ is necessarily
zero when $1\leq q<p< q^*,$ where $q^*=\frac{qn}{(n-q)}$ when $1\leq q<n$
and we define $q^*=\infty$ when $q\ge n.$ This is the first result without
additional density assumptions on the domain in question when $1\leq q\leq
n-1.$
\begin{theorem}
\label{Thm1}
Let $\Omega \subset \mathbb{R}^n$ and $1\leq q< p< q^*$ be given. If $\Omega$ is a bounded $ (W^{1,p},W^{1,q})$ -extension domain, then $|\partial \Omega|=0.$
\end{theorem}

The conclusion of Theorem \ref{Thm1} is new only when $q\le n-1$ since the case of
$q>n-1$ is already covered by \cite{PUZ}.
We do not recover the full scope of the case $q>n-1$ because our technique
does not apply capacity estimates that indeed give more refined
information when $q>n-1.$

Notice that we have not required the extension operator to be linear nor
it to be in any sense local. Even in the case of linearity, locality may
fail in a strong sense. Indeed, \cite[Theorem 1.2, Theorem 5.1]{PUZ} exhibit a bounded $(W^{1,p},W^
{1,q})$ -extension domain with a bounded linear extension operator that
extends a function $u,$ which is identically one on $B(x,r)\cap \Omega,$
to $Eu$ which vanishes on a set of positive volume contained in $B
(x,r)\cap \partial \Omega.$ Actually, ruling out this behavior is equivalent to the boundary being of volume zero \cite[Theorem 5.1]{PUZ}. Theorem
\ref{Thm1} thus shows that this kind of a pathological situation is only possible when $p$ is large in terms of $q.$

Let us close this introduction by explaining the strategy of our proof of
Theorem \ref{Thm1}. We begin by establishing a version of \eqref{sahlfors}
under our assumption $p<q^*.$ This seems to be the first result of this
type for small values of $q.$ We succeed in this by adapting a technique
from \cite{HKT} to our setting. This estimate by itself is not enough. We
continue by showing that the decay order of the volume has to exceed this
bound at almost every point of density of the  boundary. This is
established via a delicate iteration argument with the help of a set
function from \cite{PUZ}.
This set function was used in \cite{PUZ} for lower bounds of the volume,
but we succeed here to apply it towards upper bounds. Our estimates here
again strongly rely on the assumption $p<q^*.$
\section{Preliminaries}
\begin{definition}\label{de:sobolev}
Let $1\leq p<\infty.$ For $u\in L^p(\Omega)$, we say that $u$ belongs to the Sobolev space $W^{1, p}(\Omega),$ if $u$ is weakly differentiable and its weak (distributional) gradient $\nabla u$ belongs to $L^p(\Omega; \mathbb{R}^n)$. The Sobolev space $W^{1, p}(\Omega)$ is equipped with the norm
\[\|u\|_{W^{1, p}(\Omega)}:=
\left(\int_\Omega\left|u(x)\right|^p+\left|\nabla u(x)\right|^pdx\right)^{\frac{1}{p}}. 
\] 
\end{definition}
\begin{definition}
 Let $1\leq q\leq p<\infty$ and $n\geq 2$ be given. A bounded domain $\Omega\subset\mathbb{R}^n$ is said to be a $(W^{1, p}, W^{1, q})$-extension domain (Sobolev $(p,q)$-extension domain in short) if there exists a bounded extension operator 
 \begin{equation}
    E:W^{1, p}(\Omega)\to W^{1, q}(\mathbb{R}^n),  
 \end{equation}   
 that is, $Eu_{|\Omega}=u$ and there exist $C>0$ such that 
 \begin{equation*}
     \|Eu\|_{W^{1,q}(\mathbb{R}^n)}\leq C \|u\|_{W^{1,p}(\Omega)}.
 \end{equation*}
\end{definition} 

Our next lemma shows that we may assume in what follows that $1\leq q<n.$
\begin{lemma}\label{reduction}
   Let $1\leq q< p <\infty$ and $\Omega$ be a bounded Sobolev $(p,q)$-extension domain. Then $\Omega$ is also a bounded Sobolev $(p,s)$-extension domain for all $1 \leq s<q.$ 
\end{lemma}
\begin{proof}
    Let $B\subset\mathbb{R}^n$ be a ball with $\Omega \subset\subset B.$ Since $\Omega \subset \mathbb{R}^n$ is a Sobolev $(p,q)$-extension domain, for every $u\in W^{1,p}(\Omega),$ there exists $Eu \in W^{1,q}(B)$ such that $Eu{|}_{\Omega}\equiv u$ and 
    \begin{equation}\label{ineq1}
        \|Eu\|_{W^{1,q}(B)}\leq C \|u\|_{W^{1,p}(\Omega)},
    \end{equation}
    with a constant $C$ independent of $u.$\\
    By using Hölder's inequality, we get that for $1\leq s <q$
    \begin{equation}\label{holder}
        \|Eu\|_{W^{1,s}(B)}\leq |B|^{\frac{1}{s}-\frac{1}{q}}\|Eu\|_{W^{1,q}(\Omega)}.
    \end{equation}
 By combining \eqref{holder} and \eqref{ineq1}, we conclude that
 \begin{equation}\label{maineq}
     \|Eu\|_{W^{1,s}(B)}\leq C |B|^{\frac{1}{s}-\frac{1}{q}} \|u\|_{W^{1,p}(\Omega)}.
 \end{equation}
 Since $B$ is a bounded Sobolev $(s,s)$-extension domain,  by \eqref{maineq} we deduce that $\Omega$ is a bounded Sobolev $(p,s)$-extension domain. 
\end{proof}

We recall two important estimates.\\
 \textbf{Global Sobolev embedding theorem }:- Let $1\leq p<n.$ Then, there exists a constant $C= C(n,p)$ such that 
\begin{equation}\label{sobolev}
\left(\int_{\mathbb{R}^n}|u(x)|^{\frac{np}{n-p}}
dx\right)^{\frac{n-p}{pn}}\leq C\left(\int_{\mathbb{R}^n}|\nabla u(x)|^pdx\right)^{\frac{1}{p}}
\end{equation}
holds for each $u\in W^{1, p}(\mathbb{R}^n)$.\\
\textbf{Sobolev-Poincar\'e inequality}:-Let $1\leq p<n.$ Let $y\in \mathbb{R}^n$ and $r>0.$ There exists a constant $C= C(n,p)$ such that 
\begin{equation}\label{eq:locp-poin}
\left(\int_{B(y,r)}|u(x)-u_{B(y,r)}|^{\frac{np}{n-p}}dx\right)^{\frac{n-p}{np}}\leq C\left(\int_{B(y,2r)}|\nabla u(x)|^pdx\right)^{\frac{1}{p}}
\end{equation}
for each ball $B(y,r)\subset \mathbb{R}^n$ and every $u\in W^{1,p}(B(y,r)).$ Here, 
 $u_{B(y,r)}= \frac{1}{|B(y,r)|}\int_{B(y,r)}u(x) dx.$

We also need a set function associated to our extension operator. We borrow it from \cite{PUZ}, also see [\cite{1}, \cite{2}]. For convenience of the reader we review its construction and its crucial properties.
\begin{definition}
    Let $\mathcal{M}$ be a collection of open subsets of $\mathbb{R}^n.$ A function $\Phi,$
    \begin{equation*}
        \Phi :\mathcal{M}\to [0,\infty)
    \end{equation*}
    is said to be a quasiadditive set function if, for all $U_{1}\subset U_{2},
  \hspace{2mm} \text{where}\hspace{2mm} U_{1}, U_{2} \in \mathcal{M},$
    \begin{equation*}
        \Phi(U_{1})\leq \Phi(U_{2})
    \end{equation*}
    and there exists a positive constant $C$ such that for every collection of pairwise disjoint open sets $\{U_{i}\in \mathcal{M}\}_{i\in \mathbb{N}}$ we have
    \begin{equation}\label{constant}
        \sum_{i=1}^{\infty}\Phi(U_{i})\leq C\Phi\left(\bigcup_{i=1}^{\infty}U_{i}\right).
    \end{equation}
\end{definition}
The upper and lower derivatives of a quasiadditive set function,  defined on $\mathcal{M}$ containing all open balls, are
\begin{equation*}
    \overline{D\Phi}(x)= \limsup_{r\to0^{+}}{\frac{\Phi(B(x,r))}{|B(x,r)|}}\hspace{2mm}\text{and}\hspace{2mm} \underline{D\Phi}(x)= \liminf_{r\to 0^{+}}{\frac{\Phi(B(x,r))}{|B(x,r)|}}.
\end{equation*}
We formulate a result from \cite{3} and \cite{RPR}, in a convenient form.
\begin{lemma}\label{lemmafun}
    Let $\Phi$ be a quasiadditive set function defined on $\mathcal{M}$ containing all open balls. Suppose that there exists $M>0$ such that $\Phi(U)\leq M$ for all $U\in \mathcal{M}.$ Then for almost all points $x\in\mathbb{R}^n,$ the upper derivative is finite and 
        \begin{equation}\label{2.6}
            \overline{D\Phi}(x)\leq C\underline{D\Phi}(x)<\infty.
        \end{equation}
    
\end{lemma}
The constant $C$ above is the same as the one in \eqref{constant}.\\
From \eqref{2.6}, we conclude that
\begin{equation}\label{funcset}
    \Phi(B(x,r))\leq C_x |B(x,r)|,
\end{equation}
for all $0<r<r_x,$ where both $C_x$ and $r_x$ may depend on $x.$

Let $\Omega\subset \mathbb{R}^n$ be a bounded Sobolev $(p,q)$-extension domain.
 Suppose that $U\subset \mathbb{R}^n$ is an open set with $U\cap\Omega\neq \phi$. We define
\begin{equation*}
    W_{0}^{p}(U, \Omega):= \{u\in C(\Omega)\cap W^{1,p}(\Omega): u\equiv 0 \hspace{2mm} \text{on}\hspace{2mm} \Omega\setminus U\}.
\end{equation*}
For every open set $U\subset \mathbb{R}^n$ such that $U\cap\Omega\neq 0$ and each $u\in W_{0}^{p}(U,\Omega),$ we deﬁne the $q-$Dirichlet energy $\Gamma^{q}_{U}$ on $U$ with respect to the boundary value $u$ by setting
\begin{equation*}
    \Gamma_{U}^{q}(u):= \inf\left\{\left(\int_{U}|\nabla v(z)|^qdz\right)^{\frac{1}{q}}:v\in W^{1,q}(U),v|_{U\cap \Omega}\equiv u\right\}.
\end{equation*}
Then the set function $\Phi$ is defined by setting
\begin{equation}\label{setfun}
    \Phi(U):= \sup\left\{\left(\frac{\Gamma_{U}^{q}(u)}{\|u\|_{W^{p}_{0}(U\cap\Omega)}}\right)^k: u\in W_{0}^{p}(U,\Omega)\right\}
\end{equation}
with $\frac{1}{k}=\frac{1}{q}-\frac{1}{p},$ and by setting $\Phi(U)=0$ for those open sets that do not intersect $\Omega.$

The following theorems are taken from \cite{PUZ}. According to the first theorem,  $\Phi$ is a quasiadditive set function. The second theorem gives an important gradient estimate.
\begin{theorem}
Let $1\leq q<p<\infty$ and $\Omega\subset \mathbb{R}^n$ be a bounded Sobolev $(p,q)$-extension domain. Then the set function $\Phi$ deﬁned in \cref{setfun} is a bounded quasiadditive set function deﬁned on open subsets $U\subset\mathbb{R}^n.$
\end{theorem}
\begin{theorem}\label{thmset}
  Let $1 \leq q < p < \infty$. Let $\Omega \subset \mathbb{R}^n$ be a bounded Sobolev $(p,q)$-extension domain and $\Phi$ be the set function from \cref{setfun}. Then, for each ball $B := B(x, r)$ with $x \in \partial \Omega $ and every function $u \in W^{p}_{0}(B,\Omega)$  there exists a function $v\in W^{1,q}(B)$ with $v|_{B\cap \Omega}\equiv u$ and 
  \begin{equation}
\|\nabla v\|_{L^q(B)}\leq 2\Phi^{\frac{1}{k}}(B)\|u\|_{W^{1,p}(B\cap\Omega)} \hspace{2mm}\text{where}\hspace{2mm} \frac{1}{k}=\frac{1}{q}-\frac{1}{p}.
  \end{equation}
\end{theorem}
By combining Theorem \ref{thmset} and \eqref{funcset} we obtain a key estimate.
\begin{corollary}\label{corollary}
    Let $1 \leq q < p < \infty$. Let $\Omega \subset \mathbb{R}^n$ be a bounded Sobolev $(p,q)$-extension domain and $\Phi$ be the set function from \cref{setfun}. Then for every $x\in\partial\Omega$ for which $\overline{D\Phi}(x)<\infty,$ there exist $r_{x}$ and $C_{x}$ so that for each ball $B(x,r)$ with $0<r<r_x$ and every $u\in W^{p}_{0}(B(x,r),\Omega) $ there exists an extension $Eu\in W^{1,q}(B(x,r))$ of $u$ with
     \begin{equation}
\|\nabla Eu\|_{L^q(B(x,r))}\leq C_x |B(x,r)|^{\frac{1}{k}}\|u\|_{W^{1,p}(B(x,r)\cap\Omega)} \hspace{2mm}\text{where}\hspace{2mm} \frac{1}{k}=\frac{1}{q}-\frac{1}{p}.
  \end{equation}
Especially, the conclusion holds for almost every $x\in \partial\Omega.$
\end{corollary}

\section{Auxiliary Lemmas}
\begin{lemma}\label{Lbound}
    Let $\Omega $ be a bounded Sobolev $(p,q)$-extension domain with $1\leq q<p<q^*.$ Then there exists $\delta >0$ and $s=s(p,q)\geq n$ such that 
    \begin{equation*}
        |B(x,r)\cap \Omega|\geq \delta r^s,
    \end{equation*}
    for every $x\in \partial \Omega$ and all $0<r\leq \frac{\diam(\Omega)}{6}.$ When $q<n,$ we may choose $s= \frac{1}{\frac{1}{p}-\frac{1}{q^*}}$.
\end{lemma}
\begin{proof}
By Lemma \ref{reduction} we may without loss of generality assume that $1\leq q<n.$

    Fix $0<r\leq \frac{\diam(\Omega)}{6}$ and $x\in \partial \Omega.$ Then $|\Omega \setminus B(x,r)|\geq \delta_{0}>0,$ where $\delta_{0}$ neither depends on $x$ nor on $r.$\\
    Write $r_{0}=r.$ For each $j\geq 1,$ choose $r_{j}$ so that $|B(x,r_j)\cap \Omega|= \frac{|B(x,r)\cap\Omega|}{2^j}.$\\
    Fix $j\geq 1$ and define 
    \begin{equation*}
        u(y)= \max\{0,1-\frac{1}{(r_{j-1}-r_{j})}d(y,B(x,r_j))\}\hspace{2mm}\text{if}\hspace{2mm}x\in \Omega.
    \end{equation*}
    Then,
    \begin{eqnarray}
        \|u\|_{W^{1,p}(\Omega)}&\leq&  |(B(x,r_{j-1})\setminus B(x,r_j))\cap \Omega 
 |^{\frac{1}{p}}\frac{1}{(r_{j-1}-r_{j})}+ |B(x,r_{j-1})\cap \Omega |^{\frac{1}{p}},\nonumber\\
        &\leq& |B(x,r_{j-1})\cap \Omega|^{\frac{1}{p}}(r_{j-1}-r_{j})^{-1}+ |B(x,r_{j-1})\cap \Omega |^{\frac{1}{p}},\nonumber\\
        &\leq& (1+(r_{j-1}-r_{j})^{-1})|B(x,r_{j-1})\cap\Omega|^{\frac{1}{p}},\nonumber\\
        &\leq& 2(r_{j-1}-r_{j})^{-1}|B(x,r_{j-1})\cap\Omega|^{\frac{1}{p}}.
    \end{eqnarray}
    So $u\in W^{1,p}(\Omega),$
    and hence  $Eu \in W^{1,q}(\mathbb{R}^n)$ with
    \begin{equation}\label{exteq}
        \|Eu\|_{W^{1,q}(\mathbb{R}^n)}\leq C\|u\|_{W^{1,p}(\Omega)}.
    \end{equation}
    By the Global Sobolev embedding  inequality \eqref{sobolev}
    \begin{equation}\label{gpeq}
        \|Eu\|_{L^{q^*}(\mathbb{R}^n)}\leq C_{q}\|\nabla Eu\|_{L^{q}(\mathbb{R}^n)}\leq C_{q}\|Eu\|_{W^{1,q}(\mathbb{R}^n)}.
    \end{equation}
    By combining \eqref{exteq} and \eqref{gpeq}, we conclude that
    \begin{equation}\label{q*}
         \|Eu\|_{L^{q^*}(\mathbb{R}^n)}\leq C^{'}\|u\|_{W^{1,p}(\Omega)}.
    \end{equation}
    Since $u\equiv 1 $ on $B(x,r_{j}),$ we conclude that
\begin{equation}\label{poieq}
    |B(x,r_j)\cap\Omega|^{\frac{1}{q^*}}\leq  \|Eu\|_{L^{q^*}(\mathbb{R}^n)}.
\end{equation}
By using \eqref{q*}  and \eqref{poieq}, we get
\begin{eqnarray*}
    |B(x,r_j)\cap\Omega|^{\frac{1}{q^*}}\leq C_{n,q}(r_{j-1}-r_{j})^{-1}|B(x,r_{j-1})\cap\Omega|^{\frac{1}{p}}.
\end{eqnarray*}
This implies that
\begin{equation*}
  \frac{|B(x,r)\cap\Omega|^{\frac{1}{q^*}}}{2^{\frac{j}{q^*}}}(r_{j-1}-r_{j})\leq C_{n,q}\frac{|B(x,r)\cap\Omega|^{\frac{1}{p}}}{2^{\frac{(j-1)}{p}}},
\end{equation*}
that is,
\begin{equation}\label{rsum}
     (r_{j-1}-r_{j})\leq C_{n,q}\left(\frac{|B(x,r)\cap\Omega|^{\frac{1}{p}-\frac{1}{q^*}}}{2^{j(\frac{1}{p}-\frac{1}{q^*})}}\right).
\end{equation}
Since $p<q^*$, by summing \eqref{rsum} over $j,$ we arrive at
\begin{equation*}
    r\leq \sum_{j=1}^{\infty}(r_{j-1}-r_{j})\leq C_{n,q}|B(x,r)\cap\Omega|^{\frac{1}{p}-\frac{1}{q^*}}\left(\sum_{j=1}^{\infty}\frac{1}{2^{j(\frac{1}{p}-\frac{1}{q^*})}}\right).
\end{equation*}
Hence, for $\delta = \left(\frac{1}{C_{n,q}M}\right)^{\frac{1}{\frac{1}{p}-\frac{1}{q^*}}}$, where $M:= \left(\sum_{j=1}^{\infty}\frac{1}{2^{j(\frac{1}{p}-\frac{1}{q^*})}}\right),$ we have  
\begin{equation*}
|B(x,r)\cap \Omega|\geq \delta r^{\frac{1}{\frac{1}{p}-\frac{1}{q^*}}}.
\end{equation*}
\end{proof}
For a domain $\Omega$ and given $x\in\partial\Omega$, we define $A_{r_{1},r_{2},\Omega}$ as
\begin{equation*}
    A_{r_{1},r_{2},\Omega}:=\{y\in \mathbb{R}^n: r_{1}<|y-x|<r_{2}\}\cap\Omega,
\end{equation*}
where $r_{2}>r_{1}>0.$ 
\begin{lemma}\label{lemma1}
  Let  $0<c<1 $ and $1 \leq q < p < q^*.$  Let $\Omega$ be a bounded Sobolev $(p,q)$-extension domain. If, for a given $x\in \partial \Omega,$ both
   \begin{equation*}
       \limsup_{r\to 0}\frac{|B(x,r)\cap \Omega|}{r^n}=0
   \end{equation*}
   and 
   \begin{equation*}
       \overline{D\Phi}(x)<\infty
   \end{equation*}
  hold, then there exists $r_{x,c}>0$ such that
   \begin{equation*}
       \min\{|A_{r,2r,\Omega}|,|B(x,r/2)\cap\Omega|\}\leq c|B(x,r)\cap\Omega|,
   \end{equation*}
   when $0<r<r_{x,c}.$
\end{lemma}
\begin{proof}
Given $\delta>0$, our assumptions give the existence of $0< r_{x,\delta}<1$ such that
\begin{equation}\label{besti}
    |B(x,r)\cap \Omega|\leq \delta r^n,
\end{equation}
whenever $0<r<r_{x,\delta}.$\\

Consider  the function $u$ defined by
\begin{eqnarray*}
    u(y):=\begin{cases}
        1, \hspace{3mm} \text{if}\hspace{1mm} y\in B(x,r/2)\cap\Omega,\\
        -2\frac{|y-x|}{r}+2,  \hspace{4mm}\text{if}\hspace{1mm} y\in [B(x,r)\cap\Omega]\setminus B(x,r/2),\\
         0, \hspace{5mm} \text{if}\hspace{1mm} y\in \Omega\setminus B(x,r) .\\
    \end{cases}
\end{eqnarray*}
 By Corollary \ref{corollary} there exist $ r_{x}$ and $C_x$ such that, for $0<r<r_{x},$
\begin{equation}\label{puz}
    \|\nabla Eu\|_{L^q(B(x,2r))}\leq C_{x}r^{n/q -n/p}\|u\|_{W^{1,p}(B(x,r)\cap \Omega)}   
\end{equation}
for an extension $Eu$ of $u$, where $C_{x}>1.$\\
Next, by using the Sobolev-Poincar\'e inequality \eqref{eq:locp-poin}, \eqref{puz} and the definition of $u$ we deduce that
\begin{eqnarray}\label{mineq}
      \min\{|A_{r,2r,\Omega}|^{\frac{1}{q^*}},|B_{r/2, \Omega}|^{\frac{1}{q^*}}\}&\leq& C(n,q)\left(r^{(n/q-n/p)}\|u\|_{L^p(B(x,r)\cap \Omega)}+r^{(n/q -n/p)}\|\nabla u\|_{L^p(B(x,r)\cap \Omega)}\right),\nonumber \\
      &\leq&C(n,q)\left(r^{(n/q-n/p)}|B(x,r)\cap\Omega|^{\frac{1}{p}}+r^{(n/q -n/p-1)}|A_{r/2,r,\Omega}|^{\frac{1}{p}}\right),
\end{eqnarray}
when $0<r<r_x.$

Case 1:- If $\min\{|A_{r,2r,\Omega}|^{\frac{1}{q^*}}, |B(x,r/2)\cap\Omega|^{\frac{1}{q^*}}\}= |A_{r,2r,\Omega}|^{\frac{1}{q^*}}$, then we conclude that
\begin{equation}
    |A_{r,2r,\Omega}|\leq C(n,q)r^{(n/q-n/p)q^*}|B(x,r)\cap\Omega|^{\frac{q^*}{p}}+C(n,q)r^{(n/q -n/p-1)q^*}|A_{r/2,r,\Omega}|^{\frac{q^*}{p}}.
\end{equation}
In other words
\begin{eqnarray*}
     |A_{r,2r,\Omega}|&\leq& C(n,q)r^{(n/q-n/p)q^*}|B(x,r)\cap\Omega|^{\frac{q^*}{p}-1}|B(x,r)\cap\Omega|\\
     & & +C(n,q)r^{(n/q -n/p-1)q^*}|A_{r/2,r,\Omega}|^{\frac{q^*}{p}-1}|A_{r/2,r,\Omega}|.
\end{eqnarray*}
We use \eqref{besti} and the assumption that $p<q^*$ to conclude that, for $0<r<\min\{r_x,r_{x,\delta}\},$
\begin{eqnarray*}
     |A_{r,2r,\Omega}| 
     &\leq &\delta^{(\frac{q^*}{p}-1)} \cdot C(n,q)r^{(n/q -n/p)q^*}r^{n(\frac{q^*}{p}-1)}|B(x,r)\cap\Omega|\\
     & & +\delta^{(\frac{q^*}{p}-1)}\cdot C(n,q)r^{(n/q-n/p-1)q^*}r^{n(\frac{q^*}{p}-1)}|B(x,r)\cap\Omega|\\
     &\leq& \delta^{(\frac{q^*}{p}-1)}\cdot C(n,q)(r^{q*}+1)|B(x,r)\cap\Omega|\\
     &\leq& 2\delta^{(\frac{q^*}{p}-1)}\cdot C(n,q)|B(x,r)\cap\Omega|.
\end{eqnarray*}
Now, choosing $\delta^{(\frac{q^*}{p}-1)}<\frac{c}{2C(n,q)},$ we get
\begin{equation}
    |A_{r,2r,\Omega}|\leq c|B(x,r)\cap \Omega|.
\end{equation}
Case 2:-Suppose finally that $\min\{|A_{r,2r,\Omega}|^{\frac{1}{q^*}}, |B(x,r/2)\cap\Omega|^{\frac{1}{q^*}}\}= |B(x,r/2)\cap\Omega|^{\frac{1}{q^*}}$.
Then \eqref{mineq} gives
\begin{equation}
   |B(x,r/2)\cap\Omega|\leq C(n,q)r^{(n/q-n/p)q^*}|B(x,r)\cap\Omega|^{\frac{q^*}{p}}+C(n,q)r^{(n/q -n/p-1)q^*}|A_{r/2,r,\Omega}|^{\frac{q^*}{p}}.
\end{equation}
In other words
\begin{eqnarray*}
     |B(x,r/2)\cap\Omega|&\leq& C(n,q)r^{(n/q-n/p)q^*}|B(x,r)\cap\Omega|^{\frac{q^*}{p}-1}|B(x,r)\cap\Omega|\\
     & & +C(n,q)r^{(n/q -n/p-1)q^*}|A_{r/2,r,\Omega}|^{\frac{q^*}{p}-1}|A_{r/2,r,\Omega}|.
\end{eqnarray*}
We use \eqref{besti} and the assumption that $p<q^*$ to conclude that, for $0<r<\min\{r_x,r_{x,\delta}\},$
\begin{eqnarray*}
     |B(x,r/2)\cap\Omega|&\leq &\delta^{(\frac{q^*}{p}-1)} \cdot C(n,q)r^{(n/q -n/p)q^*}r^{n(\frac{q^*}{p}-1)}|B(x,r)\cap\Omega|\\
     & & +\delta^{(\frac{q^*}{p}-1)}\cdot C(n,q)r^{(n/q-n/p-1)q^*}r^{n(\frac{q^*}{p}-1)}|B(x,r)\cap\Omega|\\
     &\leq& \delta^{(\frac{q^*}{p}-1)}\cdot C(n,q)(r^{q*}+1)|B(x,r)\cap\Omega|\\
     &\leq& 2\delta^{(\frac{q^*}{p}-1)}\cdot C(n,q)|B(x,r)\cap\Omega|.
\end{eqnarray*}
 Now, choosing $\delta^{(\frac{q^*}{p}-1)}<\frac{c}{2C(q,n)},$ we get
\begin{equation}
    |B(x,r/2)\cap\Omega|\leq c|B(x,r)|.
\end{equation}
\end{proof}
Towards the upper bound for the volume, we employ the following lemma.
\begin{lemma}\label{upper}
  Let  $0<c<1/2 $ and $1\leq q < p<q^*$ be given and let $\Omega$ be a bounded Sobolev $(p,q)$-extension domain. If, for a given $x\in \partial \Omega,$ both
   \begin{equation*}
       \limsup_{r\to 0}\frac{|B(x,r)\cap \Omega|}{r^n}=0,
   \end{equation*}
   and 
   \begin{equation*}
       \overline{D\Phi}(x)<\infty
   \end{equation*}
   hold, then there exists $r_{x,c}>0$ such that
   \begin{equation*}
       |B(x,r/2)\cap\Omega|\leq c|B(x,r)\cap\Omega|,
   \end{equation*}
   when $0<r<r_{x,c}.$
\end{lemma}
\begin{proof}
    From Lemma \ref{lemma1}, we have 
    \begin{equation}\label{uslemma}
        \min\{|A_{r,2r,\Omega}|,|B(x,r/2)\cap\Omega|\}\leq c|B(x,r)\cap\Omega|.
    \end{equation}
    when $0<r\leq r_{x,c}^{'}.$
    
    Fix $0<r\leq \frac{r^{'}_{x,c}}{2}$ and suppose first that
    \begin{equation}\label{suppose}
        |A_{r,2r,\Omega}|<|B(x,r/2)\cap\Omega|.
    \end{equation}
    Then, by using \eqref{uslemma}, we get
    \begin{equation}\label{eq 3.12}
        |A_{r,2r,\Omega}|\leq c|B(x,r)\cap\Omega|.
    \end{equation}
    
    Choose $R=2r$. Again by using Lemma \ref{lemma1}, we have 
    \begin{equation}
        \min\{|A_{R,2R,\Omega}|,|B(x,R/2)\cap\Omega|\}\leq c|B(x,R)\cap\Omega|,
    \end{equation}
    since $0<R\leq r_{x,c}^{'}.$
    That is,
    \begin{equation}\label{eq 3.14}
        \min\{|A_{2r,4r,\Omega}|,|B(x,r)\cap\Omega|\}\leq c|B(x,2r)\cap\Omega|.
    \end{equation}
    
    Suppose that
    \begin{equation}
        \min\{|A_{2r,4r,\Omega}|,|B(x,r)\cap\Omega|\}=|B(x,r)\cap\Omega|.
    \end{equation}
    Then by using \eqref{eq 3.14}, we conclude that,
    \begin{equation*}
        |B(x,r)\cap\Omega|\leq c|B(x,2r)\cap\Omega|.
    \end{equation*}
    That is,
    \begin{equation}\label{eq 3.16}
        |B(x,r)\cap\Omega|\leq c(|B(x,r)\cap\Omega|+|A_{r,2r,\Omega}|).
    \end{equation}
    By using \eqref{eq 3.16} and \eqref{eq 3.12}, we deduce that
    \begin{eqnarray*}
        |B(x,r)\cap\Omega|&\leq& \left(\frac{c}{1-c}\right)|A_{r,2r,\Omega}|,\\
        &\leq&c\left(\frac{c}{1-c}\right)|B(x,r)\cap\Omega|,\\
    \end{eqnarray*}
    which cannot hold since $\frac{c^2}{1-c}<1.$ We conclude that 
    \begin{equation}\label{j=2}
         \min\{|A_{2r,4r,\Omega}|,|B(x,r)\cap\Omega|\}=|A_{2r,4r,\Omega}|.
    \end{equation}
    Hence \eqref{eq 3.14} and \eqref{eq 3.12} give
    \begin{eqnarray}\label{eq21}
        |A_{2r,4r,\Omega}|&\leq& c\left(|B(x,2r)\cap\Omega|\right),\\
        &\leq& c\left(|B(x,r)\cap\Omega| + |A_{r,2r,\Omega}|\right), \nonumber\\
        &\leq& (c+c^2)\left(|B(x,r)\cap\Omega|\right)\nonumber,\\
        &\leq& (c+1)^2\left(|B(x,r)\cap\Omega|\right).\nonumber
    \end{eqnarray}
    That is,
    \begin{equation}\label{eq 3.19}
        |A_{2r,4r,\Omega}|\leq (c+1)^2|B(x,r)\cap\Omega|. 
    \end{equation}
    
    We repeat the argument for $R=4r,$ 
    \begin{equation}\label{eq 3.20}
        \min\{|A_{4r,8r,\Omega}|,|B(x,2r)\cap\Omega|\}\leq c|B(x,4r)\cap\Omega|,
    \end{equation}
    assuming that $0<r<\frac{r_{x,c}^{'}}{8}.$

    Suppose that 
    \begin{equation*}
        \min\{|A_{4r,8r,\Omega}|,|B(x,2r)\cap\Omega|\}=|B(x,2r)\cap\Omega|.
    \end{equation*}
    Then, by \eqref{eq 3.19}, we conclude that
    \begin{eqnarray}\label{eq 3.21}
        |B(x,2r)\cap\Omega|&\leq& c|B(x,4r)\cap \Omega|,\\
        &\leq& c(|B(x,2r)\cap\Omega|+|A_{2r,4r,\Omega}|)\nonumber
    \end{eqnarray}
    By using \eqref{eq 3.21} and the first line of \eqref{eq21} we deduce that
    \begin{eqnarray*}
         |B(x,2r)\cap\Omega|&\leq& \left(\frac{c}{1-c}\right)|A_{2r,4r,\Omega}|,\\
         &\leq& c\left(\frac{c}{1-c}\right)|B(x,2r)\cap\Omega|,
    \end{eqnarray*}
    which cannot hold. Hence 
    \begin{equation}\label{eq 3.22}
         \min\{|A_{4r,8r,\Omega}|,|B(x,2r)\cap\Omega|\}=|A_{4r,8r,\Omega}|.
    \end{equation}
    By combining \eqref{eq 3.12}, \eqref{eq 3.19}, \eqref{eq 3.20} and \eqref{eq 3.22}, we obtain
    \begin{eqnarray}\label{eq 3.23}
 |A_{4r,8r,\Omega}|&\leq& c \left(|B(x,4r)\cap\Omega|\right),\\ \nonumber
    &\leq& c \left(|B(x,r)\cap\Omega| + |A_{r,2r,\Omega}|+|A_{2r,4r,\Omega}|\right),\\ \nonumber
    &\leq& (c+c^2+c(c+c^2))|B(x,r)\cap\Omega|,\\ \nonumber
    &\leq& (c+1)^3|B(x,r)\cap\Omega|.
    \end{eqnarray}
    
    We claim that, for $j\in \{1,2,\cdots,\lceil\log_{2}(\frac{r_{x,c}^{'}}{r})\rceil\}$,
    \begin{equation}\label{iterations}
        \min\{|A_{2^{j-1}r,2^{j}r,\Omega}|,|B(x,2^{j-2}r)\cap \Omega|\}= |A_{2^{j-1}r,2^{j}r,\Omega}|
    \end{equation}
    when $0<2^j r\leq r_{x,c}^{'}.$

    We will prove \eqref{iterations} by induction. 
    
    The case $j=1$ holds by our assumption \eqref{suppose} and the case $j=2$ is true by \eqref{j=2}.

    Assume that \eqref{iterations} holds for $j=k.$ We will show that it also holds for $j=k+1.$

    Suppose that for $j=k+1$, \eqref{iterations} does not hold. By Lemma \ref{lemma1}, we obtain 
    \begin{equation*}
         \min\{|A_{2^{k}r,2^{k+1}r,\Omega}|,|B(x,2^{k-1}r)\cap \Omega|\}\leq c|B(x,2^kr)\cap\Omega|,
    \end{equation*}
    which gives
    \begin{equation*}
        |B(x,2^{k-1}r)\cap\Omega|\leq c\left(\frac{c}{1-c}\right)|B(x,2^{k-1}r)\cap\Omega|,
    \end{equation*}
    which cannot hold. Hence \eqref{iterations} is true for the case $j=k+1.$
    
    By using \eqref{iterations} we deduce that, for $j\in \{1,2,\cdots,\lceil\log_{2}(\frac{r_{x,c}^{'}}{r})\rceil\}$,
    \begin{eqnarray*}
        |A_{2^{j-1}r,2^{j}r,\Omega}|&\leq& (c+1)^{j}|B(x,r)\cap\Omega|
    \end{eqnarray*}
    when $0<2^jr\leq r_{x,c}^{'}.$

For $j=\lceil\log_{2}(\frac{r_{x,c}^{'}}{r})
\rceil$, 
\begin{equation*}
    \frac{r_{x,c}^{'}}{2}\leq 2^{j-1}r\leq  r_{x,c}^{'},
\end{equation*}
and since $\sigma:=\underset{r_{x,c}^{'}\leq s\leq 2r^{'}_{x,c}}{\inf}\{|A_{s/2,s,\Omega}|\}>0,$ we get that
   \begin{eqnarray*}
       \sigma &\leq& (c+1)^{j}|B(x,r)\cap\Omega|,\\
       &\leq& C_{n}2^j r^n,\\
       &\leq& C_{n}r_{x,c}^{'}r^{n-1},
   \end{eqnarray*}
   when $0<r<r_{x,c}^{'}.$

This cannot be the case when $r\leq \left(\frac{\sigma}{c_nr_{x,c}^{'}}\right)^{\frac{1}{n-1}}.$ \\
We define
\begin{equation*}
    r_{x,c}:= \left(\frac{\sigma}{c_nr_{x,c}^{'}}\right)^{\frac{1}{n-1}}.
\end{equation*}
In conclusion, \eqref{suppose} cannot hold when $0<r< r_{x,c}$ and we conclude that  
    \begin{equation}
        |B(x,r/2)\cap\Omega|\leq c|A_{r/2,2r,\Omega}|
    \end{equation}
    whenever $0<r<r_{x,c}.$
\end{proof}
\section{Proof of the main theorem}
Suppose that $|\partial \Omega|>0.$ Then there exists $E\subset \partial \Omega,$ such that $|E|>0$ and 
\begin{equation}\label{lbp}
    \lim_{r\to 0}\frac{|\partial \Omega\cap B(x,r)|}{|B(x,r)|}=1
\end{equation}
for all $x\in E.$
Let $x\in E.$ Then \eqref{lbp} yields 
\begin{equation*}\label{phiasum}
    \limsup_{r\to 0}\frac{|B(x,r)\cap\Omega|}{|B(x,r)|}=0.
\end{equation*}

Suppose that $x \in E$ is such that 
\begin{equation}\label{case 2(1)}
    \overline{D\Phi}(x)<\infty.
\end{equation}  
Let $0<c<1/2.$ Then by Lemma \ref{upper}, we have 
\begin{equation}\label{4.3}
    |B(x,r/2)\cap \Omega|\leq c|B(x,r)\cap\Omega|.
\end{equation}
for all $0<r<\min\{\frac{\diam(\Omega)}{6},r_x\}.$\\
By iterating \eqref{4.3}, we conclude that for $j\in \mathbb{N}$
\begin{equation}\label{Literation}
    |B(x,\frac{r}{2^j})\cap\Omega|\leq c^j|B(x,r)\cap\Omega|.
\end{equation}
By Lemma \ref{Lbound}, we also have
\begin{equation}
    |B(x,R)\cap\Omega|\geq \delta R
    ^s,
\end{equation}
for all small $R.$ That is, for all large $j\in \mathbb{N}$
\begin{equation}\label{Uiteration}
    |B(x,\frac{r}{2^j})\cap\Omega|\geq \delta \left(\frac{r}{2^j}\right)^s.
\end{equation}
From \eqref{Uiteration} and \eqref{Literation}, we conclude that
\begin{equation}\label{equa}
    c^j|B(x,r)\cap\Omega|\geq \delta \cdot \left(\frac{r}{2^{j}}\right)^s.
\end{equation}
From \eqref{equa}, we get that
\begin{equation*}
    c^jC(n)r^n \geq \delta\cdot\left(\frac{r}{2^{j}}\right)^s.
\end{equation*}
Hence, for all large $j\in\mathbb{N}$
\begin{equation*}
    j(s \ln{2})\geq\ln{\left(\frac{\delta r^{s-n}}{C(n)c^j}\right)}= j\ln\left(\frac{1}{c}\right)+\ln (b), 
\end{equation*}
where $b= \frac{\delta r^{s-n}}{C(n)}.$
Since $0<c<1/2,$ 
\begin{equation*}
    j\ln{\frac{1}{c}}\geq 2|\ln b|
\end{equation*}
when $j$ is sufficiently large. In conclusion,
\begin{equation}
    j\geq \frac{j}{2s\ln 2}\ln\left({\frac{1}{c}}\right) 
\end{equation}
for all sufficiently large $j.$

Now, choosing $c$ so small that $\ln\left({\frac{1}{c}}\right)\geq 4s\ln{2}$, we get that
\begin{equation*}
    j\geq 2j
\end{equation*}
which cannot hold. Hence there is no $x\in E$ for which \eqref{case 2(1)} holds.
In conclusion, $\overline{D\Phi}(x)=\infty$ for all $x\in E,$ and \eqref{funcset} yields that the set $E$ has measure zero, and hence $|\partial \Omega|=0.$
\bibliographystyle{alpha}
\bibliography{bibliography}

\begin{thebibliography}{HKT08}

\bibitem[GM85]{GOM}
F.~W. Gehring and O.~Martio.
\newblock Quasiextremal distance domains and extension of quasiconformal mappings.
\newblock {\em J. Analyse Math.}, 45:181--206, 1985.

\bibitem[HKT08]{HKT}
Piotr Haj{\l}asz, Pekka Koskela, and Heli Tuominen.
\newblock Sobolev embeddings, extensions and measure density condition.
\newblock {\em J. Funct. Anal.}, 254(5):1217--1234, 2008.

\bibitem[Kos90]{MR1039115}
Pekka Koskela.
\newblock Capacity extension domains.
\newblock {\em Ann. Acad. Sci. Fenn. Ser. A I Math. Dissertationes}, (73):42, 1990.
\newblock Dissertation, University of Jyv\"askyl\"a, Jyv\"askyl\"a, 1990.

\bibitem[KUZ22]{PUZ}
Pekka Koskela, Alexander Ukhlov, and Zheng Zhu.
\newblock The volume of the boundary of a {S}obolev {$(p, q)$}-extension domain.
\newblock {\em J. Funct. Anal.}, 283(12):Paper No. 109703, 49, 2022.

\bibitem[MPi86]{MVP}
V.~G. Maz'ya and S.~V. Poborchi\u~i.
\newblock Extension of functions in {S}. {L}. {S}obolev classes to the exterior of a domain with the vertex of a peak on the boundary. {I}.
\newblock {\em Czechoslovak Math. J.}, 36(111)(4):634--661, 1986.

\bibitem[RR55]{RPR}
T.~Rado and P.~V. Reichelderfer.
\newblock {\em Continuous transformations in analysis. {W}ith an introduction to algebraic topology}, volume Band LXXV of {\em Die Grundlehren der mathematischen Wissenschaften in Einzeldarstellungen mit besonderer Ber\"{u}cksichtigung der Anwendungsgebiete}.
\newblock Springer-Verlag, Berlin-G\"{o}ttingen-Heidelberg, 1955.

\bibitem[Ukh99]{2}
A~Ukhlov.
\newblock Lower estimates for the norm of the extension operator of the weak differentiable functions on domains of carnot groups.
\newblock {\em Proc. of the Khabarovsk State Univ. Mathematics}, 8:33--44, 1999.

\bibitem[Ukh20]{1}
Alexander Ukhlov.
\newblock Extension operators on {S}obolev spaces with decreasing integrability.
\newblock {\em Trans. A. Razmadze Math. Inst.}, 174(3):381--388, 2020.

\bibitem[Vod87]{SKV}
S.~K. Vodop'yanov.
\newblock Existence conditions for the extendability of differentiable functions and bounds for the norm of the extension operator.
\newblock In {\em A. {H}aar memorial conference, {V}ol. {I}, {II} ({B}udapest, 1985)}, volume~49 of {\em Colloq. Math. Soc. J\'anos Bolyai}, pages 957--973. North-Holland, Amsterdam, 1987.

\bibitem[VU03]{3}
S.~K. Vodop'yanov and A.~D. Ukhlov.
\newblock Set functions and their applications in the theory of {L}ebesgue and {S}obolev spaces. {I}.
\newblock {\em Mat. Tr.}, 6(2):14--65, 2003.

\end{thebibliography}

\end{document}